\documentclass[leqno,12pt,a4paper,twoside]{article}%
\usepackage{amssymb}
\usepackage{amsfonts}
\usepackage{amsmath}
\usepackage{authblk}
\usepackage{graphicx}
\usepackage{ifthen}
\usepackage{comment}
\usepackage{xcolor}
\usepackage{enumitem}
\usepackage{lipsum,multicol}
\usepackage[T1]{fontenc}
\usepackage[english]{babel}
\setlist[itemize]{noitemsep, topsep=0pt}
\setlist[enumerate]{noitemsep, topsep=0pt}
\setlist[itemize]{leftmargin=*}
\setlist[enumerate]{leftmargin=*}
\providecommand{\U}[1]{\protect\rule{.1in}{.1in}}
\providecommand{\norm}[1]{{\left\lVert#1\right\rVert}}
\providecommand{\abs}[1]{{\left\lvert#1\right\rvert}}
\providecommand{\pr}[1]{{\left(#1\right)}} 
\providecommand{\pp}[1]{{\left[#1\right]}}
\providecommand{\set}[1]{{\left\lbrace#1\right\rbrace}}
\providecommand{\scal}[1]{{\left\langle#1\right\rangle}}

\providecommand{\AL}[1]{\textcolor{blue}{#1}}

\newcommand{\normi}[3]{\norm{#1}_{
		\ifthenelse{\equal{#2}{1}}{H_0^1\pr{\mathcal{O}_{#3}}}{%
			\ifthenelse{\equal{#2}{-1}}{H^{-1}\pr{\mathcal{O}_{#3}}}{}}}}

\makeatletter
\newcommand{\subjclass}[2][2020]{%
	\let\@oldtitle\@title%
	\gdef\@title{\@oldtitle\footnotetext{\textbf{#1 \emph{Mathematics subject classification.}} #2}}%
}
\newcommand{\keywords}[1]{%
	\let\@@oldtitle\@title%
	\gdef\@title{\@@oldtitle\footnotetext{\textbf{\emph{Key words.}} #1.}}%
}
\makeatother

\newtheorem{theorem}{Theorem}

\newtheorem{lemma}[theorem]{Lemma}

\newtheorem{problem}{Problem}
\newtheorem{proposition}[theorem]{Proposition}
\newtheorem{remark}[theorem]{Remark}

\newenvironment{proof}[1][Proof]{\noindent\textbf{#1.} }{\ \rule{0.5em}{0.5em}}

\newtheorem{ass}{Assumption}%[section]

\date{}

\begin{document}
\title{$\mathbb{L^\infty}$/$\mathbb{L}^1$ Duality Results In Optimal Control Problems}
\author[1,2]{Dan Goreac}
\author[3]{Alain Rapaport}
\affil[1]{School of Mathematics and Statistics, Shandong University,\newline
Weihai 264209, PR China}
\affil[2]{LAMA, Univ.~Gustave Eiffel, UPEM, Univ.~Paris Est Creteil, CNRS, F-77447 Marne-la-Vall\'ee, France}
\affil[ ]{E-mail: \texttt{dan.goreac@univ-eiffel.fr}\vspace{3mm}}
\affil[3]{MISTEA, Univ.~Montpellier, INRAE, Institut Agro,\newline 34060 Montpellier, France}
\affil[ ]{\texttt{E-mail: alain.rapaport@inrae.fr}}
%\affil[*]{Corresponding author, email: dan.goreac@univ-eiffel.fr}
\date{\today}
\maketitle

%%%%%%%%%%%%%%%%%%%%%%%%%%%%%%%%%%%%%%%%%%%%%%%%%%%%%%%%%%%%%%%%%%%%%%%%%

\begin{abstract}
We provide a duality result linking the value function for a control problem with supremum cost $H$ under an isoperimetric inequality $G \leq g_{max}$, and the value function for the same controlled dynamics with cost $G$ and state constraint $H \leq h_{max}$. This duality is proven for initial conditions at which lower semi-continuity of the value functions can be guaranteed, and is completed with optimality considerations. Furthermore, we provide structural assumptions on the dynamics under which such regularity can be established. As a by-product, we illustrate the partial equivalence between recent works
%papers \cite{AFG_2022} and \cite{MR_2022} 
dealing with non-pharmaceutically controlled epidemics under peak or budget restrictions.

\end{abstract}

%%%%%%%%%%%%%%%%%%%%%%%%%%%%%%%%%%%%%%%%%%%%%%%%%%%%%%%%%%%%%%%%%%%%%%%%%%%%%%%%

\section{Introduction}

In the present paper, given a controlled system set on some Euclidean space and whose solution is denoted by $x^{x_0,u}$ for initial condition $x_0$ and control $u(\cdot)$, we focus on the duality between an $\mathbb{L}^\infty$-cost problem under an isoperimetric (or area) inequality
\begin{align*}
\overline{V}(x_0;g_0):=\inf_{u(\cdot)} \sup_{t \geq 0} h(x^{x_0,u}(t)) \quad \mbox{subject to} \quad \int_0^\infty g(x^{x_0,u}(t))dt\leq g_0,
\end{align*} 
and the optimization of the total area under a state constraint, i.e. 
\begin{align*}
\underline{V}(x_0;h_0):=\inf_{u(\cdot)} \int_0^\infty g(x^{x_0,u}(t))dt \quad \mbox{subject to} \quad \sup_{t \geq 0} h(x^{x_0,u}(t)) \leq h_0.
\end{align*} 
The precise formulations and assumptions will be given in the following sections.\\
Even when isoperimetric constraints are not enforced, the $\mathbb{L}^\infty$ problem is particularly hard to tackle, especially when the optimal control is sought. When the time horizon is finite, dynamic programming approaches have been proposed (e.g. \cite{barron_ishii_89}) to characterize the value function as a viscosity solution to the associated Hamilton-Jacobi equation. On the other hand, handling a running-cost problem, even under state constraints, is, perhaps, more accessible, albeit the need for structural conditions of the domain describing the constraints (see, for instance, \cite{Soner86_1}, \cite{frankowska_plaskacz_98}, \cite{frankowska_vinter_98}, \cite{frankowska_plaskacz_00}, \cite{BFZ2011}). Furthermore, such problems fall under the realm of Pontryagin's Maximum Principle and are, therefore, more likely to provide a candidate for optimality.\\
From this point of view, a result linking the value functions of the two aforementioned problems finds its importance, especially if this is accompanied by links between the optimal controls.

With this in mind, our main result stated in Theorem \ref{ThMain} shows that the value functions $\underline{V}$ and $\overline{V}$ are (generalized) inverse of each-other. This is established under a natural lower semi-continuity assumption. Furthermore, uniqueness of the optimal control in one of the problems implies optimality of the same control for the remaining problem. This completes the duality of the two formulations. 

\medskip

The present work has been indeed motivated by two complementary contributions to the study of an epidemiological model.

\begin{enumerate}
    \item In the recent paper \cite{AFG_2022}, the authors consider, in connection with a SIR-model, the problem of minimizing a budget functional corresponding to some $g$ running cost function, while maintaining constrained the infection peak to some upper ICU-related constraint i.e. $h(s,i):=i\leq i_{\max}$ (see also \cite{Miclo}). The control parameter takes its values in some compact set $U:=\pp{0,\overline{u}}$ specifying no-confinement to maximally acceptable confinement policies $\overline{u}$. For a particular choice of the running cost $g(s,i,u):=u$, it is shown in \cite{AFG_2022} that the "greedy" control acting only as the trajectory reaches the boundary of viability kernel linked to the $i_{\max}$ restriction is the unique optimal one.  Further insights on the geometry and Hamilton-Jacobi approaches make the object of \cite{FGLX_2022}. 
    \item On the other hand, in \cite{MR_2022}, the authors consider a complementary and dual problem.  Their aim is to keep the peak of infection as low as possible given a budgetary constraint. Using Green-inspired techniques, the main result in  \cite[Proposition 2]{MR_2022} proves directly the optimality of the same type of greedy policy. The analysis is restrained to a rectangle $\pp{0,\frac{\gamma}{\beta(1-\overline{a})}}\times \pp{0,i_{\max}}$, the corner $(\frac{\gamma}{\beta(1-\overline{a})},i_{\max})$ roughly corresponding to a disease-free equilibrium (DFE) in a maximally-confined environment (corresponding to policies $\overline{u}$).
\end{enumerate}

As a by-product of our duality result, we provide, in Section \ref{SecAppl}, another proof for the optimality of the greedy control in the problem of containing the peak of infection given a budgetary constraint.  This is just an illustration of the paradigm emphasized in our opening argument: the a priori harder control problem $\overline{V}$ can be reduced to $\underline{V}$ to which Pontryagin arguments can be applied. If the optimal control is unique, then, owing to Theorem \ref{ThMain}, this is equally an optimal control for $\overline{V}$.

\medskip

The paper is organized as follows. In Section \ref{SecPrelim} we specify the dynamics, the assumptions on the data and the precise formulations for our control problems. Particular emphasis is put on the viability kernels in terms of support domains of the value functions. The main contributions of the paper are given in Section \ref{SecMain}. On the one hand, we provide, under lower semi-continuity assumptions, the duality result linking value functions and optimal controls of the two problems in Theorem \ref{ThMain}. On the other hand, we specify, in Section \ref{SubsectionSemicont}, explicit assumptions on the dynamics under which such lower semi-continuity can be achieved. The Section \ref{SecAppl} is devoted to the illustration of the implications of our main result on the SIR model with non-pharmaceutical control.
%%%%%%%%%%%%%%%%%%%%%%%%%%%%%%%%%%%%%%%%%%%%%%%%%%%%%%%%%%%%%

\section{Preliminaries}\label{SecPrelim}

\subsection{Dynamics and Assumptions}

In this work, we shall deal with a controlled dynamics
\begin{equation}
\label{Eq0}\begin{cases} \dot x(t) =f\pr{x(t),u(t)},\ \mbox{ a.e. } t\geq 0;\\
x\pr{0}=x_0\in \Omega 
\end{cases}
\end{equation}
where $\Omega$ is a subset of the $n\in\mathbb{N}^*$ -dimensional Euclidean space $\mathbb{R}^n$, and we assume $\Omega$ to have non-empty interior. We require the following standard assumptions.
\begin{ass}
\label{ass1}
$ $
\begin{enumerate}
\item The control space $U$ is a compact (subset of a) metric space. The family $\mathbb{L}^0\pr{\mathbb{R};U}$ of Borel-measurable functions $u:\mathbb{R}\longrightarrow U$ will be referred to as \emph{admissible control policies}.
\item The map %coefficient
$f:\Omega\times U\rightarrow \mathbb{R}^n$ is  continuous and $\pp{f}_1$-Lipschitz continuous in the state variable $x$ uniformly w.r.t. the control $u$ i.e.\[\pp{f}_1:=\underset{u\in U}{\sup}\ \underset{x,y\in\Omega,\ x\neq y}{\sup}\ \frac{\abs{f(x,u)-f(y,u)}}{\abs{x-y}}<+\infty.\]
\item The functions $h:\Omega\rightarrow\mathbb{R}$ and $g:\Omega\times U\rightarrow\mathbb{R}_{+}$ 
are bounded uniformly continuous and Lipschitz in the state variable $x$ uniformly w.r.t. the control $u$ i.e. \[\underset{u\in U}{\sup}\ \underset{x,y\in\Omega,\ x\neq y}{\sup}\ \frac{\abs{g(x,u)-g(y,u)}+\abs{h(x)-h(y)}}{\abs{x-y}}<+\infty.\]
We will denote by $g_\infty := \underset{(x,u) \in \Omega\times U}{\sup}g(x,u)$.
\item The set $\Omega$ is forward invariant, i.e.~any solution $x(\cdot)$ of \eqref{Eq0} with $x_0 \in \Omega$ vverifies $x(t) \in \Omega$, for any $t \geq 0$.
\end{enumerate}
\end{ass}

Under this assumption, the system \eqref{Eq0} admits an unique absolutely continuous solution denoted $x^{x_0,u}(\cdot)$ for $x_0 \in \Omega$ and  $ u \in \mathbb{L}^0\pr{\mathbb{R};U}$.

\subsection{The Control Problems}

Let us consider the extended dynamics with an additional scalar component $z$ that integrates the running cost, that is
\begin{equation}
\label{Eq0+}
\begin{cases}
\dot x(t)=f\pr{x(t),u(t)},\\
\dot z(t)=g\pr{x(t),u(t)},\\
x\pr{0}=x_0\in\Omega, \; z(0)=z_0\in\mathbb{R}
\end{cases}
% \ t \geq 0;\\
\end{equation}
whose solution is denoted $(x^{x_0,u}(\cdot),z^{x_0,z_0,u}(\cdot))$.

We recall our aim to address problems in which the running maximum is minimized while obeying an area upper bound or, vice-versa, minimize the area quantity while imposing a running constraint on the trajectories. In this context, and with respect to the newly-introduced control system, let us define the parameterized viability kernels as follows.
\begin{equation}\label{ViabKer}
\begin{array}{l}
 Viab_h(h_0):=\set{x_0\in\Omega :\ \exists u\in \mathbb{L}^0\pr{\mathbb{R}_+;U}; \; h\pr{x^{x_0,u}(t)}\leq h_0,\; \forall t\geq 0},\\
Viab_g(g_0):=\set{x_0\in\Omega :\ \exists u\in \mathbb{L}^0\pr{\mathbb{R}_+;U}; \; z^{x_0,0,u}(t)\leq g_0,\; \forall t\geq 0},\\
Viab_{hg}(h_0,g_0):=\big\{x_0\in\Omega :\ \exists u\in \mathbb{L}^0\pr{\mathbb{R}_+;U} ; \; h\pr{x^{x_0,u}(t)}\leq h_0, \; z^{x_0,0,u}(t)\leq g_0, \;  \forall t\geq 0\big\} .
\end{array}\end{equation}

\begin{remark}
$ $
\begin{enumerate}
\item The reader will have noticed that $Viab_h(h_0)$ is the largest set of initial $x_0$ for which the upper-bound $h_0$ is kept on $h$. Such sets are forward in time viable.
\item The second set $Viab_g(g_0)$ is not a viability kernel per se.  To make it one, $z$ should be considered together with the initial datum $z_0$ instead of $0$.  But,  then,  \[Viab_g(g_0)=\set{x_0\in\Omega :\ \exists u\in \mathbb{L}^0\pr{\mathbb{R}_+;U}, \exists z_0\in\mathbb{R}_+\textnormal{ s.t. }z^{x_0,z_0,u}(t)-z_0\leq g_0,\; t\geq 0}. \]
From this point of view, the initial datum $z_0$ acts as a control as well.
\item Similar assertions hold true for $Viab_{hg}$.
\item Although obvious enough, let us point out that $Viab_h,Viab_g$ considered as set-valued maps enjoy monotonicity properties (with the partial order given by the inclusion of sets).  Similar assertions can be given for the set-valued map $Viab_{hg}$ if one considers the order relation $(h_0,g_0)\prec (h_0',g_0')$ defined by $h_0\leq h_0'$ and $g_0\leq g_0'$.
\end{enumerate}
\end{remark}

We first consider the optimal problem with state constraint.
\begin{problem}\label{Prob1}
Given $x_0 \in \Omega$ and $h_{0}\in\mathbb{R}$, 
\begin{eqnarray*}
\underline {\cal P}(x_0;h_{0}): &\textnormal{minimize }& \underline{J}\pr{x_0,u}:=\int_0^{+\infty} g\pr{x^{x_0,u}(t),u(t)}dt\\
& \textnormal{over } &u\in\mathbb{L}^0\pr{\mathbb{R}_+;U},\\
& \textnormal{s.t.  }& h\pr{x^{x_0,u}(t)}\leq h_{0},\ \forall t\geq 0.
\end{eqnarray*}
\end{problem}
The value function is denoted by $\underline{V}\pr{x_0;h_{0}}$, which is set to $+\infty$ when the set of controls satisfying the constraint is empty.\\

We consider the dual problem, with integral constraint
\begin{problem}\label{Prob2}
Given $x_0 \in \Omega$ and $g_{0}\in\mathbb{R}_+$,
\begin{eqnarray*}
\overline {\cal P}(x_0;g_{0}): &\textnormal{minimize }&\overline{J}\pr{x_0,u}:=\sup_{t\geq 0}h\pr{x^{x_0,u}(t)}\\ & \textnormal{over } &u\in\mathbb{L}^0\pr{\mathbb{R}_+;U},\\
& \textnormal{s.t.  }& \int_0^{+\infty} g\pr{x^{x_0,u}(t),u(t)}dt\leq g_{0}.
\end{eqnarray*} 
\end{problem}
The value function is denoted by $\overline{V}\pr{x_0;g_{0}}$, which is set to $+\infty$ when the set of controls satisfying the constraint is empty.\\

\medskip

We shall denote in the following partial inverses of the viability kernel map $V_{hg}$ as follows
\begin{align*}
& Viab_{hg}^{-g}(x_0;h_{0}) := \{ g_{0} \in \mathbb{R}_+ \; ; \; x_0 \in Viab_{hg}(h_{0},g_{0}) \},\\
& Viab_{hg}^{-h}(x_0;g_{0}) := \{ h_{0} \in \mathbb{R} \; ; \; x_0 \in Viab_{hg}(h_{0},g_{0}) \}
\end{align*}
Then, one can formulate the following observations.
\begin{remark}
$ $
\begin{enumerate}
\item With the viability kernel notations,  and by interpreting $Viab_{hg}$ as a set-valued map, our problems amount to finding 
\begin{equation}
\label{corresV_viab}
\begin{cases}
\underline{V}\pr{x_0;h_0}&=\inf Viab_{hg}^{-g}\pr{x_0;h_0},\\
\overline{V}\pr{x_0;g_0}&=\inf Viab_{hg}^{-h}\pr{x_0;g_0}.
\end{cases}\end{equation}
This also renders coherent the fact that we have set $+\infty$ as values whenever the sets to which the $\inf$ operator is to be applied are empty. 
\item Furthermore, we have 
\begin{equation}
\label{defdom}
\begin{cases}
Dom\pr{\underline{V}\pr{x_0;\cdot}}&=\displaystyle \bigcup_{h_0\in\mathbb{R}} Dom \pr{Viab_{hg}^{-g}\pr{x_0;h_0}},\\[5mm]
Dom\pr{\overline{V}\pr{x_0;\cdot}}&=\displaystyle \bigcup_{g_0\in\mathbb{R}_+} Dom \pr{Viab_{hg}^{-h}\pr{x_0;g_0}},
\end{cases}
\end{equation}where, as usual, the domain $Dom(F)$ of a set-valued map $F:\mathbb{R}\rightsquigarrow\mathbb{R}^n$ is the family of points $\theta\in\mathbb{R}$ for which $F(\theta)\neq \emptyset$. In particular, the previously-introduced viability kernel $Viab_{hg}$ offer a complete description of the two domains. We choose to keep notations like $Dom\pr{\underline{V}\pr{x_0;\cdot}}$ only for our readers' sake.
\item The functions $\underline{V}(x_0;\cdot)$, $\overline{V}(x_0;\cdot)$ are bounded on their domains, as $g$ and $h$ are bounded functions.
\end{enumerate}
\end{remark}

\bigskip

Let us begin with some elementary and immediate properties of the two value functions.
\begin{proposition}\label{PropMonotonyRC}Let $x_0\in \Omega$.
\begin{enumerate}
\item $\overline{V}(x_0;\cdot)$ and $\underline{V}(x_0;\cdot)$ are non-increasing. 
\item If $\underline{V}(x_0;\cdot)$, resp. $\overline{V}(x_0;\cdot)$, is lower-semi-continuous, then it is right-continuous on its domain.
\end{enumerate}
\end{proposition}
\begin{proof}
Let us consider $g_{0}\leq g_{0^\prime}$ and $u(\cdot)$ a measurable control such that $\int_0^{+\infty} g\pr{x^{x_0,u}(t),u(t)}\leq g_{0}$, then one necessarily has $\int_0^\infty g\pr{x^{x_0,u}(t),u(t)}\leq g_{0^\prime}$, which implies $\overline{V}\pr{x_0;g_{0^\prime}}\leq\overline{V}\pr{x_0;g_{0}}$. A similar argument implies that $\underline{V}\pr{x_0;\cdot}$ is non-increasing.\\
By monotonicity,  if $h_0\in Dom\pr{\underline{V}(x_0;\cdot)}$, then $\left[h_0,\infty \right)\subset Dom\pr{\underline{V}(x_0;\cdot)}$, and one has
\[
\liminf_{h \to h_0+} \underline V(x_0;h) \leq \underline V(x_0;h_0).
\]
Under the further assumption that $\underline V(x_0;\cdot)$ is lower semi-continuous at $h_0$, one gets
\[
\liminf_{h \to h_0+} \underline V(x_0;h) = \underline V(x_0;h_0),
\]
that is the right continuity of $\underline V(x_0;\cdot)$ at $h_0$.
The property for $\overline V(x_0;\cdot)$ follows in the same way.
\end{proof}

%%%%%%%%%%%%%%%%%%%%%%%%%%%%%%%%%%%%%%%%%%%%%%%%%%%%%%%%%%%%%%%%%%%%

\section{The Main Results}\label{SecMain}

We first show that a duality between problems $\overline {\cal P}$ and $\underline{\cal P}$ can be established when the value functions $\underline V$, $\overline V$ are lower semi-continuous. In a second step, we give sufficient conditions for these value functions to be semi-continuous.

\subsection{The Duality Result}\label{SubsectionMainTh}

The main results of the paper which link problems $\overline{\cal P}$ and $\underline{\cal P}$ are gathered in the following statement.

\begin{theorem}\label{ThMain}
Let $x_0\in\Omega$.
%$x_0\in\mathcal{D}\pr{\underline{V}}\cap\mathcal{D}\pr{\overline{V}}$ (i.e.  such that $\max\set{\underline{V}(x_0),\overline{V}(x_0)}<\infty$). Then, 
\begin{enumerate}
\item If $\overline{V}(x_0;\cdot)$ is right continuous at $g_0 \in\mathbb{R}_+$, then, for any $h_0 \in\mathbb{R}$ such that $\underline{V}(x_0;h_0)\leq g_0$, one has $\overline{V}(x_0;g_0)\leq h_0$.\\
If $\underline{V}(x_0;\cdot)$ is right continuous at $h_0 \in\mathbb{R}$, then, for any $g_0 \in\mathbb{R}_+$ such that $\overline{V}(x_0;g_0)\leq h_0$ one has $\underline{V}(x_0;h_0)\leq g_0$.
\item If the functions $\underline{V}(x_0;\cdot)$ and $\overline{V}(x_0;\cdot)$ are lower semi-continuous on their domains \eqref{defdom},  then $\underline{V}$ and $\overline{V}$ are generalized inverse i.e. 
\begin{equation}\label{GenInv}
\begin{cases}
\underline{V}\pr{x_0;h_0}=\inf\set{g_0:\ \overline{V}(x_0;g_0)\leq h_0}, \;\; h_0 \in Dom\pr{\underline{V}\pr{x_0;\cdot}} ;\\
\overline{V}\pr{x_0;g_0}=\inf\set{h_0:\ \underline{V}(x_0;h_0)\leq g_0}, \;\;
g_0 \in Dom\pr{\overline{V}\pr{x_0;\cdot}}.
\end{cases}
\end{equation}
\item Let $h_0$ be such that $\underline{V}(x_0;h_0)<+\infty$ and $\underline{V}(x_0;\cdot)$ is lower semi-continuous.  

Posit
\[
\underline h_{0}:=\inf\set{h_0^\prime:\ \underline{V}(x_0;h_0^\prime)=\underline{V}(x_0;h_0)}, \quad 
g_0:= \underline{V}(x_0;h_0)=\underline{V}(x_0;\underline h_0).
\]
If $u^*$ is optimal for Problem $\underline{\cal P}(x_0;\underline h_0)$, then $u^*$ is optimal for Problem $\overline {\cal P}(x_0;g_0)$.
\item In particular, if Problem $\underline{\cal P}(x_0;h_0)$ admits an unique optimal control $u^*$, then $u^*$ is optimal for Problem $\overline{\cal P}(x_0;g_0)$ where
\[
g_0:=\underline{V}\pr{x_0;\underset{t\geq 0}{\sup}\ h\pr{x^{x_0,u^*}(t)}}
\]
%under the assumption that $\underline{V}\pr{x_0;\cdot}$ is lower semi-continuous.
\item Let $g_0$ be such that $\overline{V}\pr{x_0;g_0}<+\infty$ and
$\overline{V}\pr{x_0;\cdot}$ is lower semi-continuous. Posit
\[
\underline g_0=\inf\set{g_0^\prime\geq 0\ :\ \overline{V}(x_0;g_0^\prime)=\overline{V}(x_0;g_0)} , \quad h_0:=\overline{V}(x_0;g_0)=\overline{V}(x_0;\underline g_0).
\]
If $u^*$ is optimal for Problem $\overline{\cal P}(x_0;\underline g_0)$, then $u^*$ is optimal for Problem $\underline {\cal P}(x_0;h_0)$.
\item In particular, if Problem $\overline{\cal P}(x_0;g_0)$ admits an unique optimal control $u^*$, then $u^*$ is optimal for Problem $\underline{\cal P}(x_0;h_0)$ where
\[
h_0:=\overline{V}\pr{x_0;\int_0^{+\infty} g\pr{x^{x_0,u^\star}(t),u^\star(t)}dt}.
\]
\end{enumerate}
\end{theorem}
\begin{proof}
\begin{enumerate}
\item Assume $\underline{V}(x_0;h_0)\leq g_0<+\infty$ for $h_0<+\infty$. In particular, for every $\varepsilon>0$, there exists an admissible control $u^\varepsilon$ such that $\int_0^\infty g\pr{x^{x_0,u^\varepsilon}(t),u^\varepsilon(t)}dt\leq \underline{V}\pr{x_0;h_0}+\varepsilon\leq g_0+\varepsilon$ with $h\pr{x^{x_0,u^\varepsilon}(t)}\leq h_0,$ for all $t\geq 0$. Then, by definition, $\overline{V}\pr{x_0;g_0+\varepsilon}\leq h_0$.  The conclusion follows from the right-continuity of $\overline{V}(x_0;\cdot)$ at $g_0$. The remaining assertion is shown in the same way.
\item  By Proposition \ref{PropMonotonyRC} and point 1., $g_0=\underline V(x_0;h_0)$ implies $\overline V(x_0;g_0)\leq h_0$.
Then, to show \[\underline{V}\pr{x_0;h_0}=\inf\set{g_0^\prime:\ \overline{V}(x_0;g_0^\prime)\leq h_0}\] we only need to prove the inequality $\geq$.  We proceed by contradiction and assume that $\underline{V}\pr{x_0;h_0}=g_0<g^0\AL{:=}\inf\set{g_0^\prime:\ \overline{V}(x_0;g_0^\prime)\leq h_0}$.
By definition of the infimum one has $h^0:=\overline{V}(x_0;\frac{g_0+g^0}{2})> h_0$ and by monotonicity, $\overline{V}(x_0;g_0^\prime)\geq h^0,\ \forall g_0^\prime\in\pp{g_0,\frac{g_0+g^0}{2}}$. This is in contradiction with $\overline V(x_0;g_0)\leq h_0$.
The assertion concerning $\overline{V}$ is quite similar and its proof is omitted.
\item Let us fix $u^*$ as in the statement. That $u^*$ is admissible for Problem $\underline {\cal P}(x_0;h_0)$ is clear.  Indeed, by optimality of $u^*$, the area constraint is saturated i.e.  $\int_0^\infty g\pr{x^{x_0,u^*}(t),u^*(t)}dt=\underline{V}(x_0;\underline h_0)=g_0$ and, as a consequence (by Proposition \ref{PropMonotonyRC} and point 1.), one gets \[\overline{V}\pr{x_0;g_0}\leq \underline h_0.\]
Let us assume that there exists a control $\tilde{u}$ such that $\tilde{h}_0:=\overline{J}\pr{x_0,\tilde{u},g_0}<\underline h_0$.  Then $\overline{V}(x_0;g_0)\leq \tilde{h}_0$ and, thus, $\underline{V}(x_0;\tilde{h}_0)\leq g_0= \underline{V}\pr{x_0;h_0}$. This inequality is established due to the first assertion combined with the right-continuity of $\underline{V}\pr{x_0;\cdot}$ (cf. Proposition \ref{PropMonotonyRC}). By monotonicity, this can only happen when $\underline{V}(x_0;\tilde{h}_0)=\underline{V}\pr{x_0;h_{0}}$ which contradicts the choice of $h_0$.
\item When the optimal control $u^\star$ is unique, one has $\underline V(x_0;h_0)=\underline V(x_0;\inf_{t\geq 0}h\pr{x^{x_0,u^*}(t)})=g_0$ and $$\underline h_0=\inf_{t\geq 0}h\pr{x^{x_0,u^*}(t)}.$$ Then, $u^\star$ is optimal for the Problem $\underline {\cal P}(x_0;\underline h_0)$, and therefore also optimal for Problem $\overline {\cal P}(x_0;g_0)$ with $g_0=\underline V(x_0;\underline h_0)$.
\end{enumerate}
The proofs of points 5. and 6. are analogous and are omitted.
\end{proof}

\bigskip

Another remark concerns the equivalent way of writing the statements only through the viability kernels introduced in \eqref{ViabKer}.
\begin{remark}
The second assertion in Theorem \ref{ThMain} can, alternatively, be written as follows. Let $x_0\in\Omega$ be such that $\underline{V}(x_0,\cdot)$, respectively $\overline{V}(x_0,\cdot)$, is lower semi-continuous on
\[
\bigcup_{\bar{h}\in\mathbb{R}} Dom \pr{Viab_{hg}\pr{\bar{h},\cdot}^{-1}\pr{x_0}}, \mbox{ respectively } \bigcup_{\bar{g}\in\mathbb{R}_+} Dom \pr{Viab_{hg}\pr{\cdot,\bar{g}}^{-1}\pr{x_0}}
\] 
Then, one has the equivalence
\begin{equation*}
    \inf Viab_{hg}^{-g}(x_0,\bar{h}) \leq \bar g \; \Longleftrightarrow \; \inf Viab_{hg}^{-h}(x_0,\bar g) \leq \bar h.
\end{equation*}
% \begin{align*}
% & \left\{ Viab_{hg}^{-g}(x_0,\bar{h})\cap \left(-\infty,\bar{g}+\varepsilon \right]\neq \emptyset,\ \forall \varepsilon>0 \right\}\\
% & \Leftrightarrow \; \left\{  Viab_{hg}^{-h}(x_0,\bar g) \cap \left(-\infty,\bar{h}+\varepsilon \right]\neq \emptyset, \ \forall \varepsilon>0 \right\} .
% \end{align*}
Indeed, if $\bar h$ and $\bar g$ are such that $\inf Viab_{hg}^{-g}(x_0,\bar{h}) \leq \bar g$, then one has $\underline V(x_0,\bar h) \leq \bar g$ from the first equality in \eqref{corresV_viab} and 
one gets $\inf \{ g \; : \; \overline V(x_0,g) \leq \bar h \} \leq \bar g$
with the first equality in \eqref{GenInv}, which implies $\inf Viab_{hg}^{-h}(x_0,\bar g) \leq \bar h$. The reverse implication is obtained similarly using the second equalities in \eqref{corresV_viab} and \eqref{GenInv}.

It is our belief that the duality is more transparent in the initial formulation, while viability kernel formulations seem to hint to a hidden game-like behavior. In this direction, we refer the readers to \cite{B_2015}.
\end{remark}

\medskip

Finally, we obtain  as a consequence of Theorem \ref{ThMain} the following remarkable property of functions
$\overline{V}\pr{x_0;\cdot}$, $\underline{V}\pr{x_0;\cdot}$.

\begin{lemma}
Whenever $\overline{V}\pr{x_0;\cdot}$ and $\underline{V}\pr{x_0;\cdot}$ are lower semi-continuous, one has
\begin{equation}\label{constance}
\begin{cases}\overline{V}\pr{x_0;\cdot} \textnormal{ is constant on }\pp{\underline{V}\pr{x_0;\overline{V}\pr{x_0;g_0}},g_0}, \ \forall g_0 \in \mathbb{R}_+;\\ 
\underline{V}\pr{x_0;\cdot} \textnormal{ is constant on }\pp{\overline{V}\pr{x_0;\underline{V}\pr{x_0;h_0}},h_0},\ \forall h_0 \in\mathbb{R}.
\end{cases}
\end{equation}
\end{lemma}

\begin{proof}
From Proposition \ref{PropMonotonyRC}, $\overline{V}\pr{x_0;\cdot}$ and $\underline{V}\pr{x_0;\cdot}$ are everywhere right-continuous, and one gets $\overline{V}\pr{x_0;\underline{V}\pr{x_0;h_0}}\leq h_0$ and $\underline{V}\pr{x_0;\overline{V}\pr{x_0;g_0}}\leq g_0$ for any $h_0 \in \mathbb{R}$, $g_0 \in \mathbb{R}_+$.

Take $\tilde h_0:=\overline V(x_0;\underline V(x_0,h_0))$. One has then $h_0 \geq \tilde h_0$ and by monotonicity of $\underline V(x_0;\cdot)$, one gets
\[
\underline V(x_0;h_0) \leq \underline V(x_0;\tilde h_0) = \underline V(x_0;\overline V(x_0,\underline V(x_0;h_0))).
\]
On another hand, take $\tilde g_0:= \underline V(x_0;h_0)$. One has then $\underline V(x_0;\overline V(x_0;\tilde g_0)) \leq \tilde g_0$ that is
\[
\underline V(x_0;\overline V(x_0,\underline V(x_0;h_0))) \leq \underline V(x_0;h_0).
\]
One then concludes that 
\[
\underline{V}\pr{x_0;\overline{V}\pr{x_0;\underline{V}(x_0;\cdot)}}=\underline{V}(x_0;\cdot),
\]
and, in a similar way,
\[
\overline{V}\pr{x_0;\underline{V}\pr{x_0;\overline{V}(x_0;\cdot)}}=\overline{V}(x_0;\cdot).
\]
As a consequence, $\overline{V}\pr{x_0;\cdot}$, respectively $\underline{V}\pr{x_0;\cdot}$ are constant on $\pp{\underline{V}\pr{x_0;\overline{V}\pr{x_0;g_0}},g_0}$, respectively $\pp{\overline{V}\pr{x_0;\underline{V}\pr{x_0;h_0}},h_0}$.
\end{proof}

\subsection{Criteria for lower semicontinuity}\label{SubsectionSemicont}

As we have seen in the proof of Theorem \ref{ThMain} and also in Proposition \ref{PropMonotonyRC}, the
lower semi-continuity of the value functions is a crucial ingredient to obtain a duality. As a consequence, it is worthwhile to specify assumptions on the data of the problem that ensure this property.

\medskip

For this purpose, we shall consider the family of optimal control problems with discounted cost, for a discount factor $q>0$.
\begin{problem} Given $x_0 \in \Omega$ and $h_{0}\in\mathbb{R}$, 
\begin{eqnarray*}
\underline {\cal P}_q(x_0;h_{0}): &\textnormal{minimize }& \underline{J}\pr{x_0,u}:=\int_0^{+\infty} e^{-qt}g\pr{x^{x_0,u}(t),u(t)}dt\\
& \textnormal{over } &u\in\mathbb{L}^0\pr{\mathbb{R}_+;U},\\
& \textnormal{s.t.  }& h\pr{x^{x_0,u}(t)}\leq h_{0},\ \forall t\geq 0.
\end{eqnarray*}
\end{problem}
for which we denote by $\underline{V}_q\pr{x_0;h_{0}}$ the value function (set to $+\infty$ when the set of controls satisfying the constraint is empty).

\medskip

We shall also require the classical hypotheses in optimal control theory about the extended velocity set for problem $\underline {\cal P}$.
\begin{ass}
\label{ass2}
For any $x \in \Omega$, one has
\[
\bigcup_{u \in U, r\geq 0} \left[\begin{array}{c}
f(x,u)\\
g(x,u)+r
\end{array}\right] \mbox{ is closed and convex}.
\]
\end{ass}

For convenience, let us define, for any subset $L \subset \Omega$ and $(x_0,u) \in \Omega \times \mathbb{L}^0\pr{\mathbb{R}_+;U}$ the hitting time function 
\[
\tau^{x_0,u}_L := \begin{cases}
    +\infty ,& \mbox{if } x^{x_0,u}(t) \notin L , \; \forall t \geq 0,\\
     \inf \{ t ; \; x^{x_0,u}(t) \in L \} ,& \mbox{otherwise}.
\end{cases}
\]

\begin{proposition}
Let $x_0 \in \Omega$ and $h_0 \in \mathbb{R}$ such that $x_0 \in Viab_h(h_0)$.
\begin{enumerate}
\item For any $q>0$, the map
$\underline{V}_q(x_0,\cdot)$ is bounded and lower semi-continuous on $[h_0,+\infty)$.
Moreover, if $\underline{V}(x_0;\cdot)=\underset{q>0}{\sup}\ \underline{V}_q(x_0;\cdot)$, then it is also bounded and lower semi-continuous.
\item If furthermore there exists a forward invariant compact set $L\subset \Omega$ for any control $u \in \mathbb{L}^0\pr{\mathbb{R}_+;U}$ and a number $\varepsilon>0$ such that
\begin{equation}\label{ReachingTimes}
\begin{cases}
\displaystyle \min_{u\in U}g(y,u)=0,\ \forall y\in L,\\
\displaystyle T^\star := \underset{\bar{h}\in \left[h_0,h_0+\varepsilon\right)}{\sup}\ \underset{u\in \mathbb{L}^0\pr{\mathbb{R};U}}{\sup}\ \underset{x^{x_0,u} \in Viab(\bar h)}{\sup} \tau^{x_0,u}_L < +\infty,
\end{cases}
\end{equation}
then $\underline{V}\pr{x_0;\cdot}$ is bounded and lower semi-continuous on $\left[h_0,h_0+\varepsilon\right)$.
\end{enumerate}
Similar assertions hold true for $\overline{V}\pr{x_0;\cdot}$.
\end{proposition}

\begin{proof}

%\DAN{I have replaced $[h_0,\to)$ with $[h_0,\infty)$.}
Let us fix $x_0 \in \Omega$ and, for the time being, $q>0$.  For any $h_0$ such that $x_0 \in Viab_h(h_0)$, $\underline{V}_q(x_0;\cdot)$ is well defined and bounded on $[h_0,\infty)$. Moreover, 
$\underline{V}_q(x_0;\cdot)$ is non-increasing on $[h_0,\infty)$. As such, the lower semi-continuity of $\underline{V}_q(x_0;\cdot)$ at $h_0$ only needs to be shown on decreasing sequences $h_n\to h_0$ ($n\geq 1$).
Posit
\[
\underline{v}:=\underset{n\rightarrow\infty}{\lim\inf} \ \underline{V}_q\pr{x_0;h_n} < + \infty,
\]
and consider, for every $n\geq 1$, an admissible control $u_n$ such that 
\begin{align*}
& \int_0^{+\infty} e^{-qt}g\pr{x^{x_0,u_n}(t),u_n(t)}dt\leq \underline{V}_q(x_0;h_n)+\frac{1}{n};\\
& \qquad \mbox{with } \sup_{t\geq 0}h\pr{x^{x_0,u_n}(t)}\leq h_n.
\end{align*}
We then define the sequence of functions
\begin{equation*}
    v_n(t):=\int_0^{+\infty} e^{-qs}g\pr{x^{x_0,u_n}(s+t),u_n(s+t)}ds,\ t\geq 0.
\end{equation*}
Note that $v_n(\cdot)$ is the unique bounded solution of the equation
\[
\dot v_n(t)= qv_n(t)-g\pr{x^{x_0,u_n}(t),u_n(t)}, \; t \geq 0.
\]
Let us also define the set-valued map
\[
F(x,v):=\bigcup_{u\in U, \alpha \in [0,1]} \left[\begin{array}{c}
f(x,u)\\
qv-\alpha g(x,u) -(1-\alpha) g_\infty
\end{array}\right] , \; (x,v) \in \Omega \times \mathbb{R},
\]
which is Lipschitz continuous with compact convex values (from Assumptions \ref{ass1}, \ref{ass2}). Clearly, $(x^{x_0,u_n}(\cdot),v_n(\cdot))$ is solution of the differential inclusion
$(\dot x,\dot v) \in F(x,v)$.  Passing to the limit (along some subsequence), for every compact time interval $\pp{0,T}$,  $\pr{x^{x_0,u_n},v_n}$ converges uniformly to some solution $(x,v)$ of $\pr{\dot x,\dot v}\in F(x,v)$ with $x(0)=x_0$ and $v(0)=\underline v$ (as a consequence of the Theorem of compactness of solutions of differential inclusions, see e.g. \cite{Clarke}).  Furthermore, $v$ is bounded since $v_n$ are uniformly bounded by $\frac{1}{q}\norm{g}_\infty$. The procedure can be repeated to obtain a solution $ (x,v)$ defined for any $ t \in \mathbb{R}_+$. Furthermore, from Filippov selection Lemma, there exist admissible controls $(u(\cdot),\alpha(\cdot))$ such that 
\begin{equation*}\label{Estim_1}
 x(t)=x^{x_0,u}(t),\;
\dot v(t)=qv(t)-\alpha(t)g(x(t),u(t))-(1-\alpha(t))g_\infty, \; \mbox{ a.e. } t \geq 0.
\end{equation*}
Note that $v$ is a bounded solution of
\begin{equation}
    \label{vdot}
\dot v(t)=qv(t) - g(x^{x_0,u}(t),u(t)) -r(t), \quad t \geq 0,
\end{equation}
where $r$ is the bounded non-negative function
\[
r(t):=(1-\alpha(t))(g_\infty -g(x^{x_0,u}(t),u(t)), \quad t \geq 0,
\]
and that the unique bounded solution of \eqref{vdot} is given by the expression
\begin{equation}
    \label{solv}
v(t)=\int_0^{+\infty} e^{-qs} g(x^{x_0,u}(s+t),u(s+t))ds + \int_0^{+\infty} e^{-qs} r(s+t)ds.
\end{equation}

Moreover, for any $T \in (0,+\infty)$ and $p\geq 2$, the convergence of solutions $x^{x_0,u_n}$ and the continuity and boundedness of $h$ yields
\[
\underset{n\rightarrow\infty}{\lim\inf}\ \int_0^Th^p\pr{x^{x_0,u_n}(t)}\ dt\geq \int_0^Th^p\pr{x^{x_0,u}(t)}\ dt,
\]
from which one deduces
\begin{align*}
\sup_{t \in [0,T]}\ h\pr{x^{x_0,u}(t)}&= \sup_{p\geq 2}\norm{h\pr{x^{x_0,u}}}_{\mathbb{L}^p\pr{\pp{0,T};\mathbb{R}}}\\
& \leq \sup_{p\geq 2}\underset{n\rightarrow\infty}{\lim\inf}\ \norm{h\pr{x^{x_0,u_n}}}_{\mathbb{L}^p\pr{\pp{0,T};\mathbb{R}}}=\sup_{p\geq 2}\sup_{n\geq 1}\inf_{m\geq n}\ \norm{h\pr{x^{x_0,u_m}}}_{\mathbb{L}^p\pr{\pp{0,T};\mathbb{R}}}\\&\leq \sup_{n\geq 1}\inf_{m\geq n}\ \sup_{p\geq 2}\norm{h\pr{x^{x_0,u_m}}}_{\mathbb{L}^p\pr{\pp{0,T};\mathbb{R}}}=\underset{n\rightarrow\infty}{\lim\inf}\ \norm{h\pr{x^{x_0,u_n}}}_{\mathbb{L}^\infty\pr{\pp{0,T};\mathbb{R}}}\\
&\leq \underset{n\rightarrow\infty}{\lim\inf}\ \sup_{t \geq 0}\ h\pr{x^{x_0,u_n}(t)} \leq h_0,
\end{align*}
and as this last inequality is valid for any $T>0$, one deduces the inequality
\begin{equation}
    \label{h0}
\sup_{t \geq 0}\ h\pr{x^{x_0,u}(t)} \leq h_0.
\end{equation}
Finally, from \eqref{solv} and \eqref{h0} one obtains
\[
\underline v=v(0) = \int_0^{+\infty} e^{-qs} g(x^{x_0,u}(s),u(s))ds + \int_0^{+\infty} e^{-qs} r(s)ds \geq \underline V_q(x_0;h_0),
\]
that is
\[
\underset{n\rightarrow\infty}{\lim\inf} \ \underline{V}_q\pr{x_0;h_n} \geq \underline V_q(x_0;h_0),
\]
which proves the lower semi-continuity and boundedness of $\underline V_q(x_0;\cdot)$ at $h_0$. As the upper envelope
$\underset{q>0}{\sup} \underline{V}_q(x_0;\cdot)$ is lower semi-continuous, we deduce that when the value function $\underline{V}(x_0;\cdot)$ verifies $\underline{V}(x_0;\cdot)=\sup_{q>0} \underline{V}_q(x_0;\cdot)$, then it is also 
lower semi-continuous (and bounded as $g$ is bounded).

\medskip

Under assumption \eqref{ReachingTimes}, one has clearly
\[
\underline{V}_q(x_0;\bar h)=\inf_{\overset{u \in \mathbb{L}^0(\mathbb{R}_+;U)}{\sup h(x^{x_0,u}(\cdot))\leq \bar h}} \ \int_0^{T^\star} e^{-qt}g(x^{x_0,u}(t),u(t))dt,
\]
and, thus,
\[
\underline{V}(x_0;\bar h)=\inf_{\overset{u \in \mathbb{L}^0(\mathbb{R}_+;U)}{\sup h(x^{x_0,u}(\cdot))\leq \bar h}} \ \int_0^{T^\star} g(x^{x_0,u}(t),u(t))dt=\sup_{q>0} \underline{V}_q(x_0;\bar r), \; h \in [h_0,h_0+\varepsilon).
\]
\end{proof}

% \begin{remark}
% The uniform (in $u$) condition in \eqref{RechingTimes} is needed to guarantee that the problem can be set on the same time interval and has an impact on semicontinuity. However, this condition, without the $u$ and $h$ (but just at some $h_0$) guarantees that the functions $\underline{V}(x_0;\cdot)$ and $\overline{V}(x_0;\cdot)$ are proper (i.e. have non-empty domains).
% \end{remark}

\begin{comment}
\subsection{Some Hamilton-Jacobi considerations}
The reader has surely noticed that the Problem \ref{Prob1} is somewhat simpler than Problem \ref{Prob2}. Heuristically speaking, the value function $\underline{V}\pr{\cdot;h}$ satisfies the degenerated infinte-horizon Hamilton-Jacobi equation\begin{equation}
\label{HJ}
q\underline{V}(x_0;h)+\sup_{u\in U}\set{-\scal{f(x,u),\nabla_x \underline{V}\pr{x;h}}-g(x,u)}=0,
\end{equation}
with $q=0$.
Of course, the precise (viscosity) sense of this equation strongly depends on the domain $K_h$ on which it is formulated. The reader is referred to \cite{Soner86_1} (for inward conditions), to \cite{FP_1998} (for degenerate outward conditions) or \cite{BFZ2011} (without controllability conditions), etc. 
\end{comment}
\section{Illustration on an epidemiological model}\label{SecAppl}
We recall the classical epidemiological SIR model with a non-pharmaceutical control.
\begin{equation}
\label{SIR}
\begin{cases}
\dot s(t)=-\beta (1-u(t))s(t)i(t)dt,\\
\dot i(t)=\beta(1-u(t))s(t)i(t)-\gamma i(t),\\
\dot r(t)=\gamma i(t),
\end{cases}
\end{equation}
where $s(t)$, $i(t)$ and $r(t)$ stand for the densities of the susceptible, infected and recovered populations, respectively. One can check that the property $s(t)+i(t)+r(t)=1$ is satisfied for any $t\geq 0$. The control variable $u(t)$ takes values in $U=\pp{0,\overline{u}}$ with $\overline u \leq 1$.
To keep it simple, we take %consider 
here $n=2$ with state variable $(s,i) \in \Omega$ where
\[
\Omega:= \left\{ (s,i) \in \mathbb{R}^2; \; s>0, \; i>0, \; s+i \leq 1\right\} ,
\]
and consider
\[
h(s,i):=i , \quad g(s,i,u)=\lambda(s,i)u
\]
where $\lambda$ is a smooth function.
For coherence, 
instead of writing $h_0$, we will write $i^*\in\pp{0,1}$.\\ %(this being a consequence of the normalization).\\

With respect to this system and the aforementioned functionals $g$ and $h$, the papers \cite{AFG_2022} and \cite{MR_2022} offer different treatments to Problem \ref{Prob1} and Problem \ref{Prob2} respectively for the particular case when $\lambda=1$. Based on classical Pontryagin's Maximum Principle arguments, the paper \cite{AFG_2022} shows in the main result \cite[Theorem 5.6]{AFG_2022} that the unique optimal control in Problem \ref{Prob1} is the "greedy" one only acting on the boundary on the feasible region. The same type of control is shown to be optimal for Problem \ref{Prob2} in \cite[Proposition 2]{MR_2022} using alternative (Green's Theorem-based) methods (see for instance Figure \ref{Fig0} for an illustration of an optimal solution in coordinates $(s,i,z)$ with the corresponding optimal control).
\begin{figure}[!ht]
\centerline{\includegraphics[width=0.7\columnwidth]{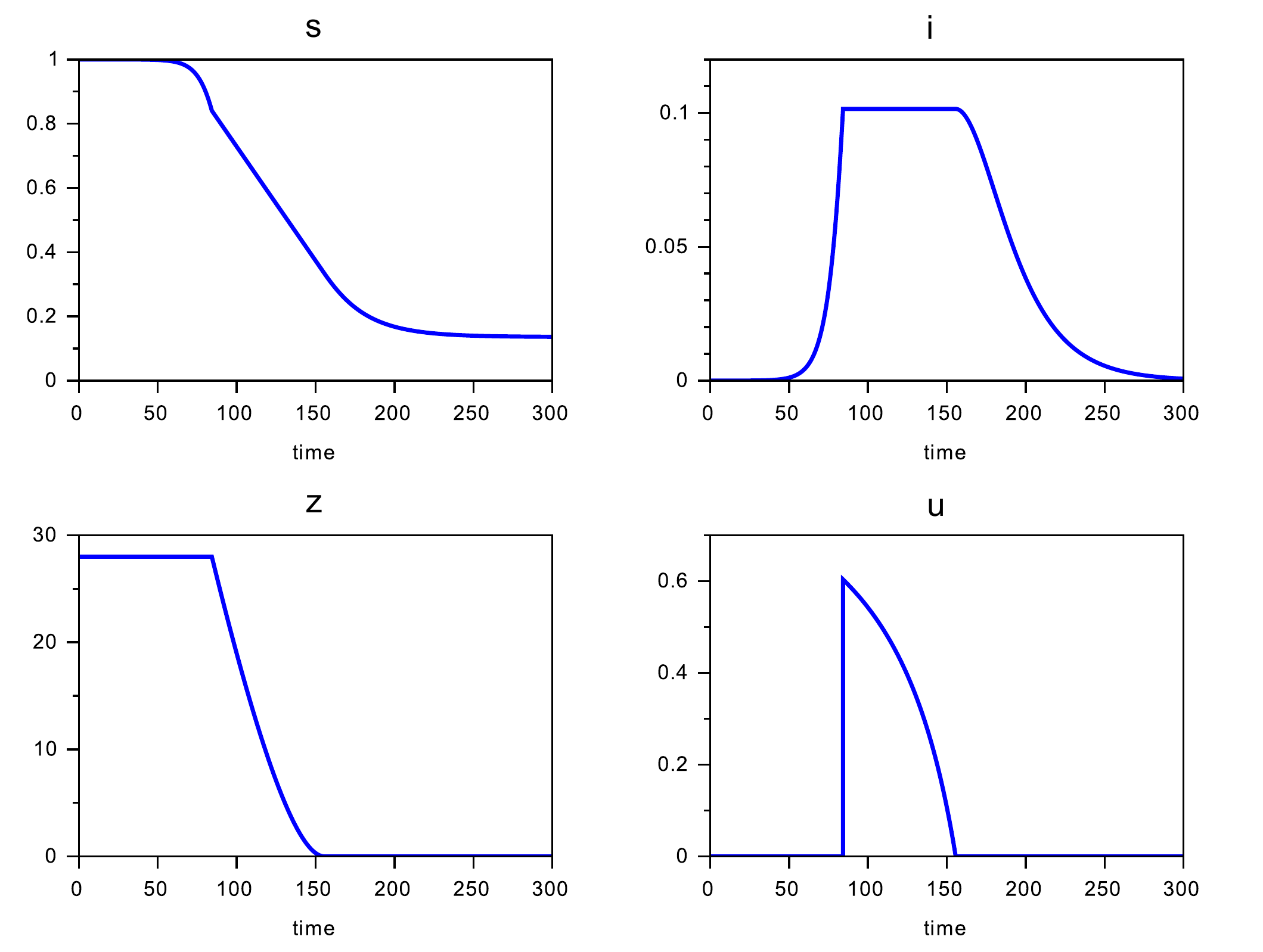}}
\caption{Example of an optimal solution for $\beta=0.21$, $\gamma=0.07$ with $i^*=0.0115$ and $g_0=28$ when $\lambda=1$ (from \cite{MR_2022}).}
\label{Fig0}
\end{figure}

We shall see in Section \ref{sec_lambda} how to generalize these results to more general functions $\lambda$.

\subsection{The geometrical structure of the domain of the value function of Problem \ref{Prob1}}
We assume that $\bar u$ is such that
\begin{equation}
\label{cond_ubar}
\bar u < 1-\frac{\gamma}{\beta}
\end{equation}
\begin{remark}
    For Problem 2, it has been shown in \cite{MR_2022} that the "null-singular-null" (NSN) strategy is such that $\max_t u(t) < 1 - \frac{\gamma}{\beta}$. This implies that this strategy is admissible when $\bar u$ verifies condition \eqref{cond_ubar}. Moreover, under this condition, the NSN strategy coincides with the greedy strategy defined in \cite{AFG_2022}, that we recall below and for which we show the optimality in Section \ref{sec_lambda}. This justifies the hypothesis \eqref{cond_ubar}.
\end{remark}

Let $i^*\in \pp{0,1}$ be fixed. Then, according to \cite[Theorem 2.3]{AFG_2022}, and provided that $i^*+\frac{\gamma}{\beta(1-\overline u)}\leq 1$ if fulfilled, one has
\begin{equation}
(s_0,i_0)\in Viab_h\pr{i^*} %\textnormal{ if and only if }
\;\Leftrightarrow\;
\begin{cases}
s_0\leq \frac{\gamma}{\beta (1-\overline{u})}\textnormal{ and }i_0\leq i^*,\\
\textnormal{ or}\\ 
s_0> \frac{\gamma}{\beta (1-\overline{u})}\textnormal{ and }i_0\leq \frac{\gamma}{\beta (1-\overline{u})}\pr{1+\log\pr{\frac{\beta (1-\overline{u})s_0}{\gamma}}}-s_0+i^*.
\end{cases}
\end{equation}
This later condition yields, in an equivalent form
\begin{equation}
\label{Kh_cond}
Dom\pr{\underline{V}(s_0,i_0);\cdot)}=\begin{cases}\left[i_0,\infty\right),\textnormal{ if }s_0\leq \frac{\gamma}{\beta (1-\overline{u})};\\
\left[i_0+s_0-\frac{\gamma}{\beta (1-\overline{u})}\pp{1+\log\pr{\frac{\beta (1-\overline{u})s_0}{\gamma}}},\infty\right),\textnormal{ otherwise}. \end{cases}
\end{equation}
%It is obvious that 
Note that this can be written in a unitary form by replacing, in the later term $s_0$ with the expression $\max\left(s_0,\frac{\gamma}{\beta (1-\overline{u})}\right)$.

For further developments, we also introduce the invariance kernel associated to $i^*$
\[
Inv_h(i^\star):=\set{(s_0,i_0)\in\Omega :\ \forall u\in \mathbb{L}^0\pr{\mathbb{R}_+;U}; \; i^{(s_0,i_0),u}(t)\leq i^\star,\; \forall t\geq 0 }
\]
and similar to $Viab_h\pr{i^*}$ (by formally taking $\overline{u}=0$), one has
\begin{equation}\label{Limax}
\pr{s_0,i_0}\in Inv_h\pr{i^*} %\textnormal{ if and only if }
\;\Leftrightarrow\;
\begin{cases}
s_0\leq \frac{\gamma}{\beta}\textnormal{ and }i_0\leq i^*,\\
\textnormal{ or}\\ 
s_0> \frac{\gamma}{\beta}\textnormal{ and }i_0\leq \frac{\gamma}{\beta}\pp{1+\log\pr{\frac{\beta s_0}{\gamma}}}-s_0+i^*.
\end{cases}
\end{equation}
Concerning the main assumptions, the reader will note that we deal with a control-affine structure here such that 
\begin{enumerate}
\item the sets $Viab_h\pr{i^*}\subset Viab_h(1)$ are compact;
\item the Assumption \ref{ass2} (convexity of the extended velocity set) is always satisfied.
\end{enumerate}
On Figure \ref{Fig1},
the (boundary of the) set $Viab_h(i^*)$ is represented by the graph of a function $\psi$ depicted in yellow, while the set $Inv_h(i^*)$ has a boundary represented in green as the graph of a function $\phi$. The intermediate set $B(i^*)$ (defined below in \eqref{B}(a)) has a blue boundary ($\partial B$), in complement of the upper barrier $i=i^*$.  Furthermore, the DFE (desease-free equilibria) for $u=0$ (resp. $u=\overline{u}$) are represented on the upper-part of the graphic.
\begin{figure}[!ht]
\centerline{\includegraphics[width=0.7\columnwidth]{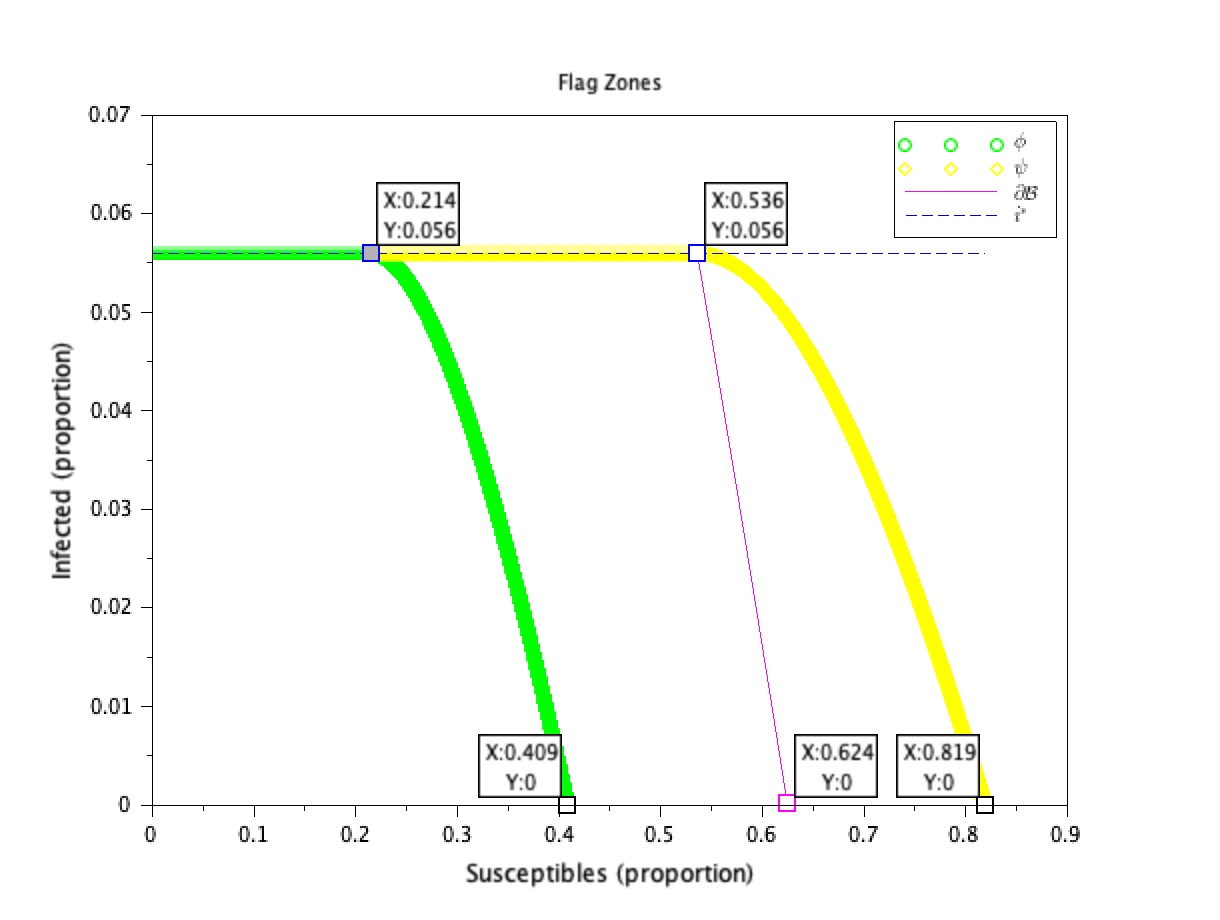}}
\caption{Geometric zones for $\beta=\frac{1}{3}$,  $1-\overline u=0.4$, $\gamma=\frac{1}{14}$, $i^*=14\times\frac{400}{100,000}$ (from \cite{AFG_2022}).}
\label{Fig1}
\end{figure}
\subsection{Regularity of the optimal cost}
Take $i^*\in Dom\pr{\underline{V}(\AL{(}s_0,i_0\AL{)};\cdot)}$ such that $i^*+\frac{\gamma}{\beta(1-\overline u)}<1$, and define the greedy feedback policy \begin{equation}
\label{OptCtrlP1}
u^*(s,i):=\begin{cases}
\overline{u},&\textnormal{ if }s>\frac{\gamma}{\beta (1-\overline{u})},\ i=i^*-s+\frac{\gamma}{\beta (1-\overline{u})}\pp{1+\log\pr{\frac{\beta (1-\overline{u})s}{\gamma}}};\\
1-\frac{\gamma}{\beta s}, &\textnormal{ if }s\in\pp{\frac{\gamma}{\beta},\frac{\gamma}{\beta(1-\overline{u})}} \textnormal{ and }i=i^*;\\
0,&\textnormal{ otherwise},
\end{cases}
\end{equation} in which non-zero action is taken only when the trajectory reaches $\partial Viab_h\pr{i^*}$.  The associated cost satisfies (see \cite[Lemma 1]{FGLX_2022})
\begin{equation}
\label{OptimValP1}
\underline{J}\pr{s_0,i_0;i^*}=\begin{cases}
0,\textnormal{ if }\pr{s_0,i_0}\in Inv_h\pr{i^*};\\[3mm]
\displaystyle  \frac{1}{\gamma i^*}\int_{\frac{\gamma}{\beta}}^{s_1(s_0,i_0;i^*)}\lambda(l,i^*)\pr{1-\frac{\gamma}{\beta l}}\,dl, \textnormal{ if } \pr{s_0,i_0}\in B\pr{i^*}\setminus Inv_h\pr{i^*};\\[4mm]
\displaystyle \frac{1}{\beta (1-\overline{u})}\int_{\frac{\gamma}{\beta (1-\overline{u})}}^{s_2\pr{s_0,i_0;i^*}}\frac{\lambda\pr{s,\pr{\theta(i^*)-s+\frac{\gamma}{\beta(1-\overline{u})}\log s}}\overline{u}}{s\pr{\theta(i^*)-s+\frac{\gamma}{\beta(1-\overline{u})}\log s}}\,ds+\\[6mm]
\hfill
\underline{J}\pr{\frac{\gamma}{\beta(1-\overline{u})},i^*;i^*}, \textnormal{ otherwise},
\end{cases}
\end{equation}
where 
\begin{equation}\label{B}
\left\{\;\begin{array}{ll}
(a): & B(i^*):=\set{(s,i)\in Viab_h(i^*):\ s+i\leq i^*+\frac{\gamma}{\beta(1-\overline{u})}+\frac{\gamma}{\beta}\log\pr{\frac{\beta(1-\overline{u})s}{\gamma}}};\\
(b): & s_1(s_0,i_0;i^*)>\frac{\gamma}{\beta}\textnormal{ is the solution of } s_1-s_0-i_0+i^*-\frac{\gamma}{\beta}\log\frac{s_1}{s_0}=0;\\
(c): & \theta(i):=i+\frac{\gamma}{\beta(1-\overline{u})}\pr{1-\log\frac{\gamma}{\beta(1-\overline{u})}};\\
(d): & s_2\pr{s_0,i_0;i^*}:=\exp\pr{\frac{\beta (1-\overline{u})}{\gamma\overline{u}}\pr{s_0+i_0-\frac{\gamma}{\beta}\log s_0-\theta(i^*)}}.
\end{array}\right.
\end{equation}
The cost $\underline{J}\pr{\frac{\gamma}{\beta(1-\overline{u})},i^*;i^*}$ used in expression \eqref{OptimValP1} depends on the function $\lambda$ and does not have necessarily an explicit expression, excepted when $\lambda$ is constant as in \cite{AFG_2022,MR_2022}.

\medskip

We claim that the following property is fulfilled.
\begin{lemma}
Fix $\pr{s_0,i_0} \in \Omega$ and $(i^n)_n$, $n \in \mathbb{N}$, a sequence decreasing to $i^*$. Then, one has
\begin{equation}\label{RC}
\lim_{n\rightarrow\AL{+}\infty}\underline{J}\pr{s_0,i_0;i^{n}}=\underline{J}\pr{s_0,i_0;i^*}.
\end{equation}
\end{lemma}
\begin{proof}
\begin{enumerate}
\item The reader will easily note that one has $Inv_h\pr{i^*}=\underset{n \in \mathbb N}{\cap}Inv_h\pr{i^n}$ (decreasing limit).
\item The same assertion holds true by defining $B\pr{i^*}$ given in \eqref{B}(a) as $B\pr{i^*}=\underset{n\in \mathbb N}{\cap}B\pr{i^n}$, where $\pr{B\pr{i^n}}_{n}$ is a non-increasing sequence.
%, i.e., $\pr{B\pr{i^n}}_{n\geq 1}$ is non-increasing and $B\pr{i^*}=\underset{n\geq 1}{\cap}B\pr{i^n}$.
\item If $(s_0,i_0)\in Inv_h\pr{i^*}$, then the equality in \eqref{RC} follows easily from the inclusion $Inv_h\pr{i^*}\subset Inv_h\pr{i^n}$ for every $n \in \mathbb N$ and by recalling that the value function is null at such points.
\item If $(s_0,i_0)\in B\pr{i^*}\setminus Inv_h\pr{i^*}$, then $(s_0,i_0)\in B\pr{i^n}$, for all $n \in \mathbb N$. If there existed a subsequence $\pr{\phi(n)}_{n}$ such that $(s_0,i_0)\in Inv_h\pr{i^{\phi(n)}}$ for any $n \in \mathbb N$, then, we would have $(s_0,i_0)\in Inv_h\pr{i^*}$ which is not the case. \\
It follows that, from some $n_0>0$ large enough and every $n\geq n_0$, one has $(s_0,i_0)\in B\pr{i^n}\setminus Inv_h\pr{i^n}$. One easily see that the function 
\[
i\mapsto s_1\pr{s_0,i_0;i}
\] is right-continuous for $i>0$, and we get equality \eqref{RC} for this framework.
\end{enumerate}
The same arguments can be applied in order to prove \eqref{RC} on $Viab_h\pr{i_{max}}\setminus B\pr{i^*}$ due to the continuity of the functions $\theta$ and $s_2$. 
\end{proof}

\bigskip

%We can gather these arguments to establish the following result.
Finally, we obtain the following result.
\begin{proposition}
%Let $s_0>0, i_0>0$ such that $s_0+i_0\leq 1$. 
Let $(s_0,i_0) \in \Omega$. Then, the value function $i^*\mapsto\underline{J}\pr{s,i;i^*}$ is right-continuous at every point 
$i^*\in Dom\pr{\underline{V}(s_0,i_0;\cdot)}$ such that 
$i^*<1-\frac {\gamma}{\beta (1-\overline u)}$, where
$Dom\pr{\underline{V}(s_0,i_0;\cdot)}$ is given by \eqref{Kh_cond}.
\end{proposition}
As a consequence,  we can apply Theorem \ref{ThMain} to show that, if the (greedy) feedback policy given in \eqref{OptCtrlP1} is the unique optimal control to the Problem \ref{Prob1}, then it is also an optimal policy for Problem \ref{Prob2} and vice-versa. 

\subsection{Differential conditions on $\lambda$ and optimality of the greedy control}
\label{sec_lambda}
Let us now emphasize the conditions needed on $\lambda$ in order to obtain optimality of the feedback control $u^*$ given in \eqref{OptCtrlP1}. We present two methods.\\

\noindent\textbf{Method I} from \cite{MR_2022} consists in writing 
\[
udt=\pr{\frac{\gamma}{\beta s}-1}\frac{ds}{\gamma i}-\frac{di}{\gamma i}.
\]
As  a consequence, one gets
\[\lambda(s,i)udt=\frac{\lambda(s,i)\pr{\frac{\gamma}{\beta s}-1}}{\gamma i}ds-\frac{\lambda(s,i)}{\gamma i}di=:P(s,i)ds+Q(s,i)di.\]
One computes\[\partial_sQ(s,i)-\partial_iP(s,i)=-\frac{\partial_s\lambda(s,i)}{\gamma i}+\pr{\frac{\gamma}{\beta s}-1}\frac{\partial_i \lambda(s,i)-\gamma\frac{\lambda(s,i)}{\gamma i}}{\gamma i}.\]
Then, the condition in \cite{MR_2022} for optimality, based on the use of Green's Theorem, amounts to imposing $\partial_sQ-\partial_iP\leq 0$ to deal with the case in which $\frac{\gamma}{\beta}<s_0\leq \frac{\gamma}{\beta (1-\bar{u})}$.  \\

\noindent\textbf{Method II} from \cite{FGLX_2022}. Let us now refer to the conditions given in \eqref{Limax}.  One writes \[\tilde{l}_1(s,i,u):=\frac{\lambda(s,i)u}{\gamma i u}.\] 
Then, 
\begin{enumerate}
\item The first condition in \cite[Eq.~(15)]{FGLX_2022} (applicable for initial conditions as specified before), requires \begin{equation}\label{EqNeed1}\frac{\lambda(s_1(s_0,i_0;i^*),i^*)}{\gamma i^*}\leq \frac{\lambda(s_0,i_0)}{\gamma i_0}.\end{equation}The reader is recalled that $(s_1(s_0,i_0;i^*),i^*)\in\set{\pr{s^{s_0,i_0,0}(t),i^{s_0,i_0,0}(t)};\ t\geq 0}$. Then, the condition \eqref{EqNeed1} is obtained if, for instance, the function $t\mapsto \phi(t):=\frac{\lambda\pr{s^{s_0,i_0,0}(t),i^{s_0,i_0,0}(t)}}{\gamma i^{s_0,i_0,0}(t)}$ is non-increasing.  One readily computes (with the obvious notation $(s,i)=\pr{s^{s_0,i_0,0},i^{s_0,i_0,0}}$)
\[\phi'(t)=\frac{1}{\gamma i}\partial_s\lambda(s,i)\pr{-\beta si}+\pp{\frac{\partial_i \lambda(s,i)}{\gamma i} -\frac{\lambda(s,i)}{\gamma i^2}}\pr{\beta s-\gamma}i=\beta si\pr{\partial_sQ-\partial_i P}.\]The latter quantity is non-positive as soon as $\partial_sQ-\partial_i P\leq 0$.
We conclude that, in the case where $s_0\leq \frac{\gamma}{\beta (1-\bar{u})}$, the "$0$-singular arc-$0$" control is optimal with the two methods. This is, of course, a vivid illustration of our main result in the present paper.
\item The second condition in \eqref{Limax} amounts to have \begin{equation}\label{EqNeed2}\frac{\lambda(s_2(s_0,i_0;i^*),\theta^*-s_2(s_0,i_0;i^*)+\frac{\gamma}{\beta (1-\bar{u})}\log s_2(s_0,i_0;i^*))}{\gamma \pr{\theta^*-s_2(s_0,i_0;i^*)+\frac{\gamma}{\beta (1-\bar{u})}\log s_2(s_0,i_0;i^*)}}\leq \frac{\lambda(s_0,i_0)}{\gamma i_0}.\end{equation}
As before, $(s_2(s_0,i_0;i^*),i^*)$ belongs
%the arguments in the left-hand member belong 
to the reachable set \[\set{\pr{s^{s_0,i_0,0}(t),i^{s_0,i_0,0}(t)}:\ t\geq 0}.\]
Reasoning as we have done for case 1., the condition \eqref{EqNeed2} follows from the same condition $\partial_sQ-\partial_i P\leq 0$ (on a different part of the space as this time $s\geq \frac{\gamma}{\beta(1-\bar{u})}$).
The reader is invited to note that under the condition \eqref{EqNeed2}, owing to the result on duality, we are able to extend the optimality result in \cite{MR_2022} to any admissible $(s_0,i_0)$ beyond the DFE (disease free-equilibria) for $\tilde{\beta} :=\beta (1-\bar{u})$-contact driven SIR (i.e. extend it to configurations for which $s_0>\frac{\gamma}{\tilde{\beta}}=\frac{\gamma}{\beta (1-\bar{u})}$. 
\end{enumerate}
\color{black}
\bibliographystyle{abbrv}
\bibliography{biblio}

\def\cprime{$'$}
\begin{thebibliography}{10}

\bibitem{AFG_2022}
F.~Avram, L.~Freddi, and D.~Goreac.
\newblock {Optimal control of a SIR epidemic with ICU constraints and target
  objectives}.
\newblock {\em Applied Mathematics and Computation}, 418:126816, 2022.

\bibitem{barron_ishii_89}
E.~Barron and H.~Ishii.
\newblock The {B}ellman equation for minimizing the maximum cost.
\newblock {\em Nonlinear Anal., Theory Methods Appl.}, 13(9):1067--1090, 1989.

\bibitem{B_2015}
P.~Bettiol.
\newblock State constrained $l^\infty$ optimal control problems interpreted as
  differential games.
\newblock {\em Discrete and Continuous Dynamical Systems}, 35(9):3989--4017,
  2015.

\bibitem{BFZ2011}
O.~Bokanowski, N.~Forcadel, and H.~Zidani.
\newblock Deterministic state-constrained optimal control problems without
  controllability assumptions.
\newblock {\em ESAIM: COCV}, 17(4):995--1015, 2011.

\bibitem{Clarke}
F.~Clarke.
\newblock {\em {Optimization and Nonsmooth Analysis}}.
\newblock {SIAM Classics in Applied Mathematics}, 1990.

\bibitem{frankowska_plaskacz_98}
H.~Frankowska and S.~Plaskacz.
\newblock {Semi-continuous solutions of {H}amilton-{J}acobi-{B}ellman equations
  with state constraints}.
\newblock {\em Differential Inclusions and Optimal Control, vol. 2, Lecture
  Notes in Nonlinear Anal.}, pages 145--161, 1998.

\bibitem{frankowska_plaskacz_00}
H.~Frankowska and S.~Plaskacz.
\newblock Semicontinuous solutions of hamilton-jacobi-bel lman equations with
  degenerate state constraints.
\newblock {\em J. Math. Anal. Appl.}, 251:818--838, 2000.

\bibitem{frankowska_vinter_98}
H.~Frankowska and R.~Vinter.
\newblock Existence of neighbouring trajectories: applications to dynamic
  programming for state constraints optimal control problems.
\newblock {\em Journal of Optimization Theory and Applications}, 104(1):20--40,
  2000.

\bibitem{FGLX_2022}
L.~Freddi, D.~Goreac, J.~Li, and B.~Xu.
\newblock {SIR Epidemics with State-Dependent Costs and ICU Constraints: A
  Hamilton--Jacobi Verification Argument and Dual LP Algorithms}.
\newblock {\em Applied Mathematics {\&} Optimization}, 86(2):23, Jul 2022.

\bibitem{Miclo}
L.~Miclo, D.~Spiro, and J.~Weibull.
\newblock {Optimal epidemic suppression under an ICU constraint: An analytical
  solution}.
\newblock {\em Journal of Mathematical Economics}, 101:102669, 2022.

\bibitem{MR_2022}
E.~Molina and A.~Rapaport.
\newblock {An optimal feedback control that minimizes the epidemic peak in the
  SIR model under a budget constraint}.
\newblock {\em Automatica}, 146:110596, 2022.

\bibitem{Soner86_1}
H.~M. Soner.
\newblock Optimal control with state-space constraint. {I}.
\newblock {\em SIAM J. Control Optim.}, 24(6):552--561, 1986.

\end{thebibliography}
\end{document}